\documentclass{elsarticle}
\usepackage{amssymb,amsmath}
\usepackage{amsthm, dsfont}
\usepackage{hyperref}
\usepackage[all]{xy}
\usepackage[active]{srcltx}
\usepackage[short,nodayofweek]{datetime}

\journal{Journal of Algebra}
\def\NN{{\mathbb N}}

\def\QQ{{\mathbb Q}}

\def\G{{\mathcal G}}
\def\H{{\mathcal H}}
\def\I{{\mathcal I}}

\def\fH{{\mathfrak H}}
\def\fT{{\mathfrak T}}

\def\m{{\mathfrak m}}

\def\tR{\tilde{R}}

\def\id{{\rm id}}
\def\ID{{\rm ID}}
\def\cat#1{{\sf #1}}

\def\vect#1{\text{\boldmath $#1$\unboldmath}} 
\def\ie{i.\,e.}
\def\isom{\cong}

\def\restr#1{|_{#1}}
\def\ideal{\unlhd}
\def\gener#1{\left\langle #1 \right\rangle}

\DeclareMathOperator{\Ker}{Ker}
\DeclareMathOperator{\Ima}{Im}
\DeclareMathOperator{\Coker}{Coker}
\DeclareMathOperator{\Hom}{Hom}
\DeclareMathOperator{\Aut}{Aut}
\DeclareMathOperator{\GL}{GL}
\DeclareMathOperator{\Mat}{Mat}
\DeclareMathOperator{\Gal}{Gal}
\DeclareMathOperator{\Isom}{Isom}
\DeclareMathOperator{\Quot}{Quot}
\DeclareMathOperator{\Spec}{Spec}
\newcommand{\uGal}{\underline{\Gal}}
\newcommand{\uIsom}{\underline{\Isom}}

\def\markdef{\bf }

\theoremstyle{plain}
\newtheorem{thm}{Theorem}[section]
\newtheorem{cor}[thm]{Corollary}
\newtheorem{lem}[thm]{Lemma}
\newtheorem{prop}[thm]{Proposition}

\theoremstyle{definition}
\newtheorem{defn}[thm]{Definition}
\newtheorem{exmp}[thm]{Example}
\newtheorem{rem}[thm]{Remark}

\newtheoremstyle{Acknowledgements}% name
  {}% {\topsep}%      Space above
    {}% {\topsep}%      Space below
     {}%         Body font
     {}%         Indent amount (empty = no indent, \parindent = para indent)
    {\bfseries}% Thm head font
    {}%        Punctuation after thm head
     {.5em}%     Space after thm head: " " = normal interword space;
    {\thmname{#1}\thmnumber{ }\thmnote{ (#3)}}%  Thm head spec (can beleft
\theoremstyle{Acknowledgements}

\begin{document}

\begin{frontmatter}

%\title[PV-theory over rings]{Picard-Vessiot theory of differentially simple rings}
\title{Picard-Vessiot theory of differentially simple rings} 

\author{Andreas Maurischat}
\address{\rm {\bf Andreas Maurischat}, Lehrstuhl A f\"ur Mathematik, RWTH Aachen 
University, Templergraben 55, 52056 Aachen, Germany }
\ead{andreas.maurischat@matha.rwth-aachen.de}

% AMS-Classification
%\subjclass[2010]{12H20, 13B05}

% Date
%\date{\today, \currenttime}
%\date{\today}

%-----------------------------------------------------------  
\begin{abstract}
In Picard-Vessiot theory, the Galois theory for linear differential equations, the Picard-Vessiot ring plays an
important role, since it is the Picard-Vessiot ring which is a torsor (principal homogeneous space)
for the Galois group (scheme). Like fields are simple rings having only $(0)$ and $(1)$ as ideals, the Picard-Vessiot ring is a differentially simple ring, i.e.~a differential ring having only $(0)$ and $(1)$ as differential ideals.
Having in mind that the classical Galois theory is a theory of extensions of fields, i.e.~of simple rings, it is quite natural to ask whether one can also set up a Picard-Vessiot theory where the base is not a differential field, but more general a differentially simple ring. It is the aim of this article to give a positive answer to this question, i.e.~to set up such a differential Galois theory.
\end{abstract}
% Keywords


%-----------------------------------------------------------  

\end{frontmatter}

%\maketitle
%---------------------------------------------------------------------
% Beginn des eigentlichen Artikels
%---------------------------------------------------------------------

\section{Introduction}

Differential Galois theory -- and also difference Galois theory -- is a generalisation of classical Galois theory to transcendental extensions. Instead of polynomial equations, one considers differential resp.~difference equations, and even iterative differential equations in positive characteristic.
One branch of differential Galois theory is Picard-Vessiot theory, the study of linear differential equations. In this case the Galois group turns out to be a linear algebraic group or more generally an affine group scheme of finite type over the field of constants. As one considers solution fields of the polynomial equations in classical Galois theory, there are solution fields in Picard-Vessiot theory.
This however is a flaw in the theory because the Galois group scheme does not act on the solution field but on a subring whose field of fractions is the solution field. This subring, called Picard-Vessiot ring, even has nice properties. For instance, this ring is differentially simple, i.e.~it does not have any differential ideals apart from $\{0\}$ and the ring itself, just as fields are simple rings.
This raises the question whether the Picard-Vessiot theory should be a theory of extensions of differentially simple rings rather than fields.

In this article, we therefore generalise (and modify) the existing Picard-Vessiot theory to extensions of  differentially simple rings.

In positive characteristic, differential rings are replaced by rings with an iterative derivation, since the constants of a derivation are too big in positive characteristic (every $p$-th power is a constant in characteristic $p$). On the other hand, in characteristic zero there is a one-to-one-correspondence between derivations and iterative derivations, as long  as the ring contains the rational numbers -- a condition which is  fulfilled for differentially simple rings. Hence, a Picard-Vessiot theory  of differentially simple rings in characteristic zero is a special case of a Picard-Vessiot theory for iterative differentially simple rings in arbitrary characteristic.

The whole paper therefore is formulated for iterative differential rings (ID-rings for short) of arbitrary characteristic and iterative differential modules (ID-modules for short).
Apart from small modifications, the proofs in this paper also work in more general settings, e.g.~for rings with several commuting iterative derivations. 

The main results are an existence result for Picard-Vessiot rings even if the constants are not algebraically closed (see Thm.~\ref{thm:pv-ring-exists}), the existence of a Galois group scheme $\G$ for a Picard-Vessiot ring $R$ (see Cor.~\ref{cor:galois-group-scheme}), as well as a Galois correspondence between intermediate Picard-Vessiot rings $S\subseteq T\subseteq R$ and closed normal subgroup schemes $\H$ of $\G$.

\smallskip

In Section 2, the  basic notation, some basic examples and first properties of ID-rings are given. We procede in Section 3 with giving some properties of ID-simple rings. Section 4 is dedicated to ID-modules over ID-simple rings. In our definition, ID-modules are finitely generated as modules over the ring.
A main difference to ID-modules over ID-fields is that  ID-modules over ID-rings need not be free as modules. However, we show (see Thm.~\ref{thm:M-projective}) that all ID-modules over ID-simple rings are projective as modules.
This also implies (see Cor.~\ref{cor:submodule-is-fin-gen}) that ID-stable submodules are again finitely generated modules, although our ID-rings are not assumed to be Noetherian.

In Section 5, we are studying Picard-Vessiot rings for ID-modules over a fixed ID-simple ring $S$. We will define them as minimal ID-simple solution rings (Def.~\ref{def:pv-ring}) and show later on that this definition coincides with the usual one (given for example in \cite[Sect.~3]{bhm-mvdp:ideac}) if the ID-module is free as an $S$-module.\\
Furthermore, we give criteria for the existence of a Picard-Vessiot ring. Namely, if the ID-ring $S$ has a $C_S$-rational point (where $C_S$ denotes the field of constants), then for every ID-module $M$ there exists a Picard-Vessiot ring (cf.~\ref{thm:pv-ring-exists}).
As in the case of an ID-field as base ring, we show in Thm.~\ref{thm:c-alg-closed} for a general ID-simple ring that a Picard-Vessiot ring always exists and is even unique up to ID-isomorphisms, if the constants $C_S$ are algebraically closed.

Section 6 is dedicated to the Galois group of a Picard-Vessiot extension and the Galois correspondence. A Galois correspondence is obtained between intermediate PV-rings and closed normal subgroup schemes of the Galois group. Arbitrary closed subgroup schemes don't fit into the picture here, because the factor schemes will not be affine, and hence the corresponding ``ID-scheme'' would not be the spectrum of a ring.

\section{Basic notation}\label{sec:notation}

All rings are assumed to be commutative with unit and different from $\{0\}$.
We will use the following notation (see also \cite{am:igsidgg}). An iterative 
derivation on a ring $R$ is a homomorphism of rings $\theta:R\to
R[[T]]$, such that $\theta^{(0)}=\id_R$ and for all $i,j\geq
0$, $\theta^{(i)}\circ \theta^{(j)}=\binom{i+j}{i}\theta^{(i+j)}$,
where the maps $\theta^{(i)}:R\to R$ are defined by
$\theta(r)=:\sum_{i=0}^\infty \theta^{(i)}(r)T^i$.
The pair $(R,\theta)$ is then called an ID-ring and
$C_R:=\{ r \in R\mid \theta(r)=r\}$ is called the {\markdef ring of
  constants } of $(R,\theta)$. An ideal $I\ideal R$ is called an
 {\markdef ID-ideal} if $\theta(I)\subseteq I[[T]]$ and $R$ is
 {\markdef ID-simple} if $R$ has no
 ID-ideals apart from $\{0\}$ and $R$. An ID-ring which is a field is called an {\markdef ID-field}.
Iterative derivations are extended to localisations by
 $\theta(\frac{r}{s}):=\theta(r)\theta(s)^{-1}$ and to tensor products
 by 
$$\theta^{(k)}(r\otimes s)=\sum_{i+j=k} \theta^{(i)}(r)\otimes
\theta^{(j)}(s)$$ 
for all $k\geq 0$.

A homomorphism of ID-rings $f:S\to R$ is a ring homomorphism $f:S\to R$ s.t. $\theta_R^{(n)}\circ f=f\circ \theta_S^{(n)}$ for all $n\geq 0$.
If $\tR$ is an ID-ring extension of $R$. Then an element $r\in \tR$ is called {\markdef ID-finite} over $R$ if the $R$-submodule of $\tR$ generated by $\{ \theta^{(k)}(r) \mid k\geq 0\}$ 
is finitely generated. 

\medskip

An {\markdef ID-module} $(M,\theta_M)$ over an ID-ring $R$ is a finitely generated $R$-module $M$ together with an iterative derivation $\theta_M$ on $M$, i.e.~an additive map $\theta_M:M\to M[[T]]$ such that $\theta_M(rm)=\theta(r)\theta_M(m)$, $\theta_M^{(0)}=\id_M$ and
$\theta_M^{(i)}\circ \theta_M^{(j)}=\binom{i+j}{i}\theta_M^{(i+j)}$ for all $i,j\geq 0$. 

A subset $N\subseteq M$ of an ID-module $(M,\theta_M)$ is called {\markdef ID-stable}, if $\theta^{(n)}(N)\subseteq N$ for all $n\geq 0$. An {\markdef ID-submodule} of $(M,\theta_M)$ is an ID-stable $R$-submodule $N$ of $M$ which is finitely generated as $R$-module.\footnote{Since $R$ may not be Noetherian, $R$-submodules of finitely generated modules may not be finitely generated.}
For an ID-module $(M,\theta_M)$ and an ID-stable $R$-submodule $N\subseteq M$, the factor module $M/N$ is again an ID-module with the induced iterative derivation.

The free $R$-module $R^n$ is an example of an ID-module over an ID-ring $R$ with  iterative derivation given componentwise. An ID-module $(M,\theta_M)$ over $R$ is called {\markdef trivial} if $M\isom R^n$ as ID-modules, i.e.~if $M$ has a basis of constant elements.

For ID-modules $(M,\theta_M)$, $(N,\theta_N)$, the {\markdef direct sum} $M\oplus N$ is an ID-module with iterative derivation given componentwise, and the {\markdef tensor product} $M\otimes_R N$ is an ID-module with iterative derivation $\theta_\otimes$ given by 
$$\theta_\otimes^{(k)}(m\otimes n):=\sum_{i+j=k} \theta_M^{(i)}(m)\otimes \theta_N^{(j)}(n)\qquad \text{for all }k\geq 0.$$

For ID-modules $(M,\theta_M)$, $(N,\theta_N)$, a {\markdef morphism} $f:(M,\theta_M)\to(N,\theta_N)$ of ID-modules is a homomorphism $f:M\to N$ of the underlying modules such that $\theta_N^{(k)}\circ f=f\circ \theta_M^{(k)}$ for all $k\geq 0$. For a morphism $f:(M,\theta_M)\to(N,\theta_N)$,  the kernel $\Ker(f)$ and the image $\Ima(f)$ are ID-stable $R$-submodules of $M$ resp.~$N$. The image $\Ima(f)$ is indeed an ID-submodule, since it is isomorphic to $M/\Ker(f)$. Also the cokernel $\Coker(f)$ is an ID-module.

\begin{exmp}\label{ex:ID-rings}
\begin{enumerate}
\item For any field $C$ and $R:=C[t]$, the homomorphism of $C$-algebras
$\theta_t:R \to R[[T]]$
given by $\theta_t(t):=t+T$ is an iterative derivation on $R$ with field of constants $C$. This
iterative derivation will be called the {\markdef iterative derivation with respect to~$t$}. $R$ is indeed an ID-simple ring, since for any polynomial $0 \ne f\in R$ of degree $n$, $\theta^{(n)}(f)$ equals the leading coefficient of $f$, and hence is invertible in $R=C[t]$.
\item\label{item:der by t} For any field $C$,  $C[[t]]$ also is an ID-ring with the iterative derivation with respect to $t$, given by $\theta_t(f(t)):=f(t+T)$ for $f\in C[[t]]$. The constants of $(C[[t]],\theta_t)$ are $C$, and $(C[[t]],\theta_t)$ also is ID-simple, since for 
$f=\sum_{i=n}^\infty a_i t^i\in C[[t]]$ with $a_n\ne 0$, one has
$$\theta_t^{(n)}(f)= \sum_{i=n}^\infty a_i \binom{i}{n} t^{i-n}\in C[[t]]^\times.$$
Hence, every non-zero ID-ideal contains a unit.
This ID-ring will play an important role, since every ID-simple ring can be ID-embedded into $C[[t]]$ for an appropriate field $C$ (comp.~Thm.~\ref{thm:embedding}).
\item For any ring $R$, there is the {\markdef trivial} iterative derivation on $R$ given by
$\theta_0:R\to R[[T]],r\mapsto r\cdot T^0$. Obviously, the ring of constants of $(R,\theta_0)$ is
$R$ itself.
\item Given a differential ring $(R,\partial)$ containing the rationals (i.e.~a $\QQ$-algebra $R$ with a derivation $\partial$), then $\theta^{(n)}:=\frac{1}{n!}\partial^n$ defines an iterative derivation on $R$. On the other hand, for an iterative derivation $\theta$, the map $\theta^{(1)}$ always is a derivation. Hence, differential rings containing $\QQ$ are special cases of ID-rings.\\
Since for a differentially simple ring in characteristic zero, its ring of constants always is a field (same proof as for ID-simple rings), we see that  the Picard-Vessiot theory for ID-simple rings, we provide here, contains a Picard-Vessiot theory for differentially simple rings in characteristic zero as a special case.
\end{enumerate}
\end{exmp}

\begin{rem}
We will not assume our rings to be Noetherian. Hence, for an ID-module over an ID-ring $R$, there might exist  $R$-submodules which are stable under the iterative derivation, but are not finitely generated as $R$-module, and therefore are not ID-modules in our definition. In particular, the kernel of a morphism of ID-modules is not an ID-module in general.\\
Another problem that might occur is concerned with ID-finiteness of elements. In general, the set of ID-finite elements in a ring extension $\tR$ does not have any extra structure (sums and products of ID-finite elements may be not ID-finite). Furthermore, there might be elements $r\in R$ which are not ID-finite over $R$, since the ideal generated by all $\theta^{(k)}(r)$ ($k\geq 0$) does not need to be finitely generated.
For ID-simple rings, however, both points will work out fine as we will see in Cor.~\ref{cor:submodule-is-fin-gen}, resp.~in Prop.~\ref{prop:ID-finite}, and Cor.~\ref{cor:ID-finite-subalgebra}.
\end{rem}

\begin{prop}\label{prop:constants-of-triv-ext}
Let $(R,\theta)$ be an ID-ring, $C:=C_R$ its ring of constants, and let $D/C$ be a ring extension such that $D$ is free as $C$-module. Moreover, let $D$ be equipped with the trivial iterative derivation $\theta_D(d)=d\in D[[T]]$ for all $d\in D$.
Then the constants of $R\otimes_C D$ are exactly the elements $1\otimes d$, $d\in D$. 
\end{prop}

\begin{proof}
By definition all elements $1\otimes d$ are constant. For proving that there are no others, let $(d_i)_{i\in I}$ be a basis of $D$ as $C$-module, and consider an arbitrary constant element $\sum_{i\in I} r_i\otimes d_i\in R\otimes_C D$ (almost all $r_i$ equal to $0$). Then for all $k\geq 0$, 
$$0=\theta^{(k)}(\sum_{i\in I} r_i\otimes d_i)=\sum_{i\in I} \theta^{(k)}(r_i)\otimes d_i.$$
Therefore, all $r_i$ are constant, i.e.~$r_i\in C$.\\
Hence, $\sum_{i\in I} r_i\otimes d_i= 1\otimes (\sum_{i\in I}  r_id_i)$.
\end{proof}

%----------------------------------------
%----------------------------------------
\section{Properties of ID-simple rings}

Since our basic objects will be ID-simple rings, we will now summarize some properties of ID-simple rings:
\begin{prop}\label{prop:first-properties}
Let $(S,\theta)$ be an ID-simple ring. Then
\begin{enumerate}
\item $S$ is an integral domain.
\item The field of fractions of $S$ has the same constants as $S$.
\item The ring of constants of $S$ is a field.
\end{enumerate}
\end{prop}

\begin{proof}
i) and ii) are proved in \cite[Lemma 3.2]{bhm-mvdp:ideac}. However, ii) also follows as a special case of Prop.~\ref{prop:ID-finite}, since constants are ID-finite elements. Part iii) follows from ii), since the inverses of constants are constants, and hence the ring of constants of an ID-field is indeed a field.
\end{proof}

\begin{prop}\label{prop:ID-finite}
Let $(S,\theta)$ be an ID-simple ring. Then an element $x\in \Quot(S)$ is ID-finite over $S$ if and only if $x\in S$.
\end{prop}

\begin{proof}
If $x\in S$, then $I:=\gener{\theta^{(n)}(x) \mid n\in\NN}_S$ is an ID-ideal of $S$, hence $I=\{0\}$ or $I=S=\gener{1}_S$. In both cases $I$ is finitely generated, and hence $x$ is ID-finite.

Now assume $x\in \Quot(S)$ is ID-finite over $S$, so by definition the $S$-module $M:=\gener{\theta^{(n)}(x) \mid n\in\NN}_S\subseteq \Quot(S)$ is finitely generated. $M$ is also stable under the iterative derivation, as is easily verified by calculation.
 The ideal $I:=\{ s\in S\mid s m\in S \ \forall m\in M\}$ is non-zero, since it contains the product of the denominators of generators of $M$.\\
We will show that $I$ is an ID-ideal. From this the claim follows, since by ID-simplicity of $S$, this will imply $I=S$, and hence $1\cdot x\in S$.\\
For all $s,m\in \Quot(S)$, $n\in \NN$ the equation
$$\theta^{(n)}(s\cdot m)=\sum_{i+j=n} \theta^{(i)}(s)\theta^{(j)}(m)=\theta^{(n)}(s)\cdot m + \sum_{i=0}^{n-1} \theta^{(i)}(s)\theta^{(n-i)}(m)$$
holds. In particular, for all $s\in I$, $m\in M$ we inductively obtain for all~$n\in \NN$:
$$\theta^{(n)}(s)\cdot m=\theta^{(n)}(\underbrace{s\cdot m}_{\in S}) - \sum_{i=0}^{n-1} \underbrace{\theta^{(i)}(s)}_{\in I\text{ by ind.hyp.}}\underbrace{\theta^{(n-i)}(m)}_{\in M} \in S,$$
and hence, $\theta^{(n)}(s)\in I$. Therefore, $I$ is an ID-ideal.
\end{proof}

\begin{prop}\label{prop:ideal-bijection}
Let $(S,\theta)$ be an ID-simple ring and $C:=C_S$ its field of constants, and let $D$ be a finitely generated $C$-algebra equipped with the trivial iterative derivation. Then there is a bijection
\begin{eqnarray*}
\I(D) &{\xymatrix{\ar@{<->}[r]&}} & \I^{\ID}(S\otimes_C D) \\
I & {\xymatrix{\ar@{|->}[r]&}} & S\otimes_C I \\
 J\cap (1\otimes_C D) & {\xymatrix{\ar@{<-|}[r]&}}& J
\end{eqnarray*}
\noindent between the ideals of $D$ and the ID-ideals of $S\otimes_C D$.
\end{prop}

\begin{proof}
cf. \cite[Lemma 10.7]{am:gticngg}.
\end{proof}

\begin{thm}\label{thm:embedding}
Let $(S,\theta)$ be an ID-simple ring,  $\m\ideal S$ a maximal ideal, and $C=S/\m$ the residue field. Then $(S,\theta)$ can be embedded into $(C[[t]],\theta_t)$ as ID-ring.
\end{thm}

\begin{proof}
The iterative derivation $\theta$ induces an injective ring homomorphism $\tilde{\theta}:S\to S[[t]], x\mapsto \sum_{n=0}^\infty \theta^{(n)}(x)t^n$, and it is easy to check, that $\tilde{\theta}$ is indeed an ID-homomorphism $(S,\theta)\to (S[[t]],\theta_t)$ where $\theta_t$ denotes the iterative derivation with respect to $t$ (comp.~Example \ref{ex:ID-rings}\ref{item:der by t}). Since $\m[[t]]$ is an ID-ideal of $S[[t]]$ and $S$ is ID-simple, also $\bar{\theta}:S\to (S/\m)[[t]]=C[[t]]$ is injective which is the desired ID-embedding.
\end{proof}

%----------------------------------------
%----------------------------------------
\section{ID-modules over ID-simple rings}

Throughout the section, let $C$ denote an arbitrary field, and let $(S,\theta)$ denote an ID-simple ring with field of constants $C$.

\begin{rem}
For a finitely generated $S$-module $M$, the following conditions are equivalent (see \cite[Section II.5.2, Theorem 1]{nb:cac1-7}):
\begin{enumerate}
\item $M$  is projective.
\item $M$  is finitely presented and locally free in the weaker sense, i.e.~for every prime ideal $P\ideal S$ the localisation $M_P=S_P\otimes_S M$ is a free $S_P$-module.
\item $M$  is locally free in the stronger sense: there exist $x_1,\dots, x_r\in S$, generating the unit ideal, such that for each $i$, $M[\frac{1}{x_i}]$ is a free $S[\frac{1}{x_i}]$-module.
\end{enumerate}
Furthermore, Cartier showed in \cite[Appendice, Lemme 5]{pc:qrdga}, that  the condition ``finitely presented" in ii) is superfluous if $S$ is an integral domain.

Since,  ID-simple rings are integral domains by Proposition \ref{prop:first-properties}, in our situation the conditions \emph{projective}, \emph{locally free in the weaker sense} and \emph{locally free in the stronger sense} are equivalent for finitely generated modules.
\end{rem}

\begin{lem}\label{lem:M-free}
Assume that $S$ is a local ring and let $\m$ denote the maximal ideal. Let $(M,\theta_M)$ be an ID-module over $(S,\theta)$. Then $M$ is a free $S$-module.
\end{lem}

\begin{proof}
Let $\{x_1,\dots, x_n\}$ be a minimal set of generators of $M$, and assume that this set is $S$-linearly dependent, i.e.~there are $r_i\in S$ (not all of them equal to $0$) such that $r_1x_1+\cdots +r_nx_n=0$.
Since $S$ is ID-simple, for each $r_i$ there is some $k_i\in \NN_0$ such that $\theta^{(k_i)}(r_i)\not\in \m$, i.e.~$\theta^{(k_i)}(r_i)\in S^\times$. Take $k\in\NN_0$ maximal such that for all $j<k$ and all $i=1,\dots, n$: $\theta^{(j)}(r_i)\in \m$. W.l.o.g. $\theta^{(k)}(r_1)\in S^\times$.
Then one obtains:
\begin{eqnarray*}
0 &=& \theta_M^{(k)}(r_1x_1+\cdots +r_nx_n)=\sum_{i=1}^n \left( \sum_{j=0}^k \theta^{(j)}(r_i)\theta_M^{(k-j)}(x_i) \right) \\
&\equiv & \sum_{i=1}^n \theta^{(k)}(r_i)x_i \mod{\m M} 
\end{eqnarray*}
Since $\theta^{(k)}(r_1)$ is invertible, this implies $x_1\in \gener{x_2,\dots,x_n}+\m M$, hence
$ \gener{x_2,\dots,x_n}+\m M=M$, and by Nakayama's lemma $\gener{x_2,\dots,x_n}=M$ contradicting the assumption that $\{x_1,\dots, x_n\}$ was minimal.

Hence, $\{x_1,\dots, x_n\}$ is linearly independent, and therefore a basis of $M$.
\end{proof}

\begin{thm}\label{thm:M-projective}
If $(M,\theta_M)$ is an ID-module over $(S,\theta)$, then $M$ is a projective $S$-module.
\end{thm}

\begin{proof}
For every prime ideal $P\ideal S$ the localisation $S_P$ is a local ring and an ID-simple ring, and $M_P=S_P\otimes_S M$ is an ID-module over $S_P$.
By the previous lemma, $M_P$ is free for all $P$, i.e.~$M$ is locally free in the weaker sense, hence projective.
\end{proof}

\begin{cor}\label{cor:submodule-is-fin-gen}
Let  $(M,\theta_M)$ be an ID-module over $(S,\theta)$. Then every ID-stable $S$-submodule of $M$ is a finitely generated $S$-module, and hence an ID-submodule of $M$.
\end{cor}

\begin{proof}
Let $N$ be an ID-stable $S$-submodule of $M$, then the factor module $M/N$ is again an ID-module, and hence a projective $S$-module by the previous theorem. Therefore, the exact sequence of $S$-modules
$$0\to N \to M\to M/N \to 0$$
splits. Hence, $N$ is isomorphic to a factor module of $M$ and therefore finitely generated.
\end{proof}

\begin{cor}\label{cor:ID-finite-subalgebra}
Let $(R,\theta)$ be an ID-ring extension of $(S,\theta)$. Then the set of elements in $R$ which are ID-finite over $S$ is an $S$-subalgebra of $R$.
\end{cor}

\begin{proof}
By Prop.~\ref{prop:ID-finite}, all elements in $S$ are ID-finite over $S$. So it remains to show that for ID-finite elements $x,y\in R$ also $x+y$ and $x\cdot y$ are ID-finite.
Since $x$ and $y$ are ID-finite over $S$, the $S$-modules $\gener{\theta^{(n)}(x) \mid n\in\NN}_S$ and $\gener{\theta^{(n)}(y) \mid n\in\NN}_S$ are finitely generated. But then also 
$M:= \gener{\theta^{(n)}(x) \mid n\in\NN}_S+\gener{\theta^{(n)}(y) \mid n\in\NN}_S$ is finitely generated as well as $N:= \gener{\theta^{(n)}(x)\theta^{(m)}(y) \mid n,m\in\NN}_S$. Therefore, $M$ and $N$ are ID-modules over $S$. Using additivity of the $\theta^{(n)}$ resp.~the generalized Leibniz rule, one obtains that $\gener{\theta^{(n)}(x+y) \mid n\in\NN}_S$ and $\gener{\theta^{(n)}(x\cdot y) \mid n\in\NN}_S$ are ID-stable $S$-submodules of $M$ resp.~of $N$, and hence by Cor.~\ref{cor:submodule-is-fin-gen}, they are both finitely generated as $S$-modules. Therefore $x+y$ and $x\cdot y$ are ID-finite over $S$.
\end{proof}

We end the considerations on the structure of ID-modules by looking at the special case of the ID-simple ring $(S,\theta)=(C[[t]],\theta_t)$ (comp.~Example \ref{ex:ID-rings}\ref{item:der by t}).

\begin{thm}\label{thm:trivial-over-ct}
Let $C$ be a field. Then every ID-module over $(C[[t]],\theta_t)$ is trivial.
\end{thm}

\begin{proof}
Let $(M,\theta_M)$ be an ID-module over $(C[[t]],\theta_t)$. Since, $C[[t]]$ is a local ring, $M$ is a free $C[[t]]$-module by Lemma \ref{lem:M-free}. Hence, let $\vect{b}=(b_1,\dots, b_r)$ be a basis of $M$ and $A(t,T)\in \Mat_{r\times r}( C[[t]][[T]])$ be such that $\theta_M(\vect{b})= \vect{b}A(t,T)$.\,\footnote{When we apply $\theta$ resp.~$\theta_M$ to a tuple or a matrix, it is meant to apply $\theta$ resp. $\theta_M$ to each entry. Then the equation has to be read as a matrix identity, i.e.~$\theta_M(b_i)=\sum_{j=1}^r b_jA(t,T)_{ji}$ for $A(t,T)_{ji}$ being the $(j,i)$-th entry of $A(t,T)$.}
 Since $\theta^{(0)}=\id_M$, one has $A(t,0)=\mathds{1}_r\in \GL_r(C[[t]])$ which implies that $A(t,T)$ is invertible, i.e.~$A(t,T)\in \GL_r( C[[t]][[T]])$. Therefore, also $Y(t):=A(t,-t)\in \Mat_{r\times r}( C[[t]])$ is invertible, since $Y(0)=A(0,0)=\mathds{1}_r\in \GL_r(C)$. We claim that $\vect{b}Y(t)$ is a basis of constant vectors in $M$, and hence $M\isom S^r$ as ID-modules.\\
Since, $\theta_M(\vect{b}Y(t))=\theta_M(\vect{b})\theta_t(Y(t))=\vect{b}A(t,T)Y(t+T)$, we have to show that
$$Y(t)=A(t,T)Y(t+T).$$
Since the iteration rule $\theta_M^{(i)}\circ \theta_M^{(i)}=\binom{i+j}{i}\theta_M^{(i+j)}$ holds, one has the following commutative diagram
%\begin{center}
$$\xymatrix{
M \ar[r]^{{}_U\theta_M} \ar[dr]_{{}_{T+U}\theta_M} & M[[U]] \ar[d]^{\theta_M[[U]]} \\
 & M[[T,U]]
}$$
%\end{center}
(which indeed is equivalent to the iteration rule; cf.~\cite[\S 27]{hm:crt}). Here ${}_U\theta_M$ and ${}_{T+U}\theta_M$ are the iterative derivations on $M$ with $T$ replaced by $U$ resp.~by $T+U$, i.e.~
${}_U\theta_M:M\xrightarrow{\theta_M} M[[T]] \xrightarrow{T\mapsto U} M[[U]]$. The map $\theta_M[[U]]$ denotes the extension of $\theta_M$ to $M[[U]]$ by applying $\theta_M$ to each coefficient, i.e.~ 
$\theta_M[[U]]\left(\sum_{i=0}^\infty m_iU^i\right):=\sum_{i=0}^\infty \theta_M(m_i)U^i\in M[[T,U]]$.

Applying this to our setting, we obtain
\begin{eqnarray*}
\vect{b}A(t,T+U) &=& {}_{T+U}\theta_M(\vect{b})=\theta_M[[U]]\left({}_U\theta_M(\vect{b})\right)\\
&=& \theta_M[[U]]\left(\vect{b}A(t,U)\right)=\vect{b}A(t,T)A(t+T,U),
\end{eqnarray*}
hence $A(t,T+U)=A(t,T)A(t+T,U)$. Specializing $U$ to $-t-T$, we finally get
$$Y(t)=A(t,-t)=A(t,T)A(t+T,-t-T)=A(t,T)Y(t+T).$$
\end{proof}

%----------------------------------------
%----------------------------------------
\section{Picard-Vessiot rings}

Throughout the section, let $C$ denote an arbitrary field, $(S,\theta)$ an ID-simple ring with $C_S=C$, and let $(M,\theta_M)$ denote an ID-module over $S$.

\begin{defn}\label{def:pv-ring}
A {\markdef solution ring} for $M$ is an ID-simple ring $0\ne (R,\theta_R)$ together with a homomorphism of ID-rings $f:S\to R$ s.t.
\begin{enumerate}
\item[(i)] $C_R=f(C)$,
\item[(ii)] $R\otimes_S M$ is a trivial ID-module over $R$.
\end{enumerate}
A {\markdef Picard-Vessiot ring} (PV-ring) for $M$ is a \emph{minimal} solution ring $0\ne (R,\theta_R)$, i.e.~if $0\ne (\tilde{R},\theta_{\tilde{R}})$ with $\tilde{f}:S\to \tilde{R}$ is another solution ring, then any ID-homomorphism $g:\tilde{R}\to R$ (if it exists) is an isomorphism.
\end{defn}

\begin{rem}
\begin{enumerate}
\item Since the kernel of an ID-homomorphism is an ID-ideal, and $S$ is ID-simple, the homomorphism $f$ is always injective. Therefore, we can view any solution ring $R$ as an extension of $S$, and we will omit the homomorphism $f$.
\item Assume that $M$ is a free $S$-module with basis $\vect{b}=(b_1,\dots, b_r)$, and $R$ is a solution ring for $M$,  then there is a matrix $Y\in \GL_r(R)$ s.t. $\vect{b}Y$ is a basis of constant elements in $R\otimes_S M$. Such a matrix will be called a {\markdef fundamental solution matrix} for $M$ (with respect to $\vect{b}$).
\end{enumerate}
\end{rem}

The next proposition implies that in case of an ID-module $M$ which is free as $S$-module, our definition of PV-ring coincides with the usual one given for example in \cite[Sect.~3]{bhm-mvdp:ideac}  (if the constants are algebraically closed) resp.~in \cite[Def.~2.3]{am:igsidgg}.

\begin{prop}\label{prop:generated-by-fsm}
Assume that $M$ is free as an $S$-module, and let $R$ be a solution ring for $M$. Then there is a unique Picard-Vessiot ring $\tilde{R}$ inside $R$. This is the $S$-subalgebra of $R$ generated by the coefficients of a fundamental solution matrix and the inverse of its determinant.
\end{prop}

\begin{proof}
Let $Y\in \GL_r(R)$ be a fundamental solution matrix for $M$ with respect to a basis $\vect{b}$. Then any other fundamental solution matrix for $M$  in $\GL_r(R)$ is obtained as $F\cdot Y\cdot G$ for $F\in \GL_r(S)$ and $G\in \GL_r(C)$ (base change in $M$ resp. $C_{R\otimes_S M}$). Hence, all fundamental solution matrices generate the same $S$-subalgebra $\tilde{R}=S[Y,\det(Y)^{-1}]$ of $R$. And since every solution ring has to contain a fundamental solution matrix, there is no solution ring strictly contained in $\tilde{R}$. It remains to show that $\tR$ is ID-simple.

Let $C[Z,\det(Z)^{-1}]$ be the group ring of $\GL_{r,C}$ equipped with the trivial iterative derivation, and let $S[X,\det(X)^{-1}]$ be an ID-extension of $S$ with $\theta(X):=\theta(Y)Y^{-1}X$, i.e. $X$ is also a fundamental solution matrix for $M$. Then the $R$-linear map
$$\lambda: R\otimes_S S[X,\det(X)^{-1}]\to R\otimes_C C[Z,\det(Z)^{-1}], X\mapsto Y\otimes Z$$
is an ID-isomorphism. Furthermore $\tR\isom S[X,\det(X)^{-1}]/\tilde{I}$ for an appropriate ID-ideal $\tilde{I}$.
By Proposition \ref{prop:ideal-bijection}, $\lambda(R\otimes \tilde{I})=R\otimes \tilde{J}$ for some ideal $\tilde{J}\ideal  C[Z,\det(Z)^{-1}]$. Hence, we obtain an $R$-linear ID-isomorphism
$$\bar{\lambda}:R\otimes_S \tR \to  R\otimes_C C[Z,\det(Z)^{-1}]/\tilde{J}.$$
Now let $I\ideal \tR$ be a maximal ID-ideal, then again by Proposition \ref{prop:ideal-bijection}, there is an ideal \mbox{$J\ideal C[Z,\det(Z)^{-1}]$} containing $\tilde{J}$, s.t. $\bar{\lambda}(R\otimes I)=R\otimes J$, and we again obtain an $R$-linear ID-isomorphism $R\otimes_S (\tR/I)\to R\otimes_C C[Z,\det(Z)^{-1}]/J$.
The restriction of this map to $1\otimes_S (\tR/I)$ therefore, gives an ID-monomorphism
$\tR/I\to  R\otimes_C C[Z,\det(Z)^{-1}]/J$. Choosing a maximal ideal $\m$ of $C[Z,\det(Z)^{-1}]/J$, and letting $D:=(C[Z,\det(Z)^{-1}]/J)/\m$ (a finite extension of $C$), leads to $\tR/I\to  R\otimes_C D$ which is again injective, since $\tR/I$ is ID-simple and the kernel is an ID-ideal. On the other hand, its $D$-linear extension $\tR/I\cdot D\to  R\otimes_C D$ has image $\tR\otimes_C D$ (since $Y=\bar{\lambda}(Y\bar{Z}^{-1})$).
But this means that the transcendence degree of $\tR/I$ over $S$ has to be at least as big as the one of $\tR$ over $S$. Hence, $I$ has to have height $0$, i.e. $I=0$ since, $\tR\subseteq R$ is an integral domain.
\end{proof}

\begin{thm}\label{thm:explicite-pv-ring}
Let $M$ be an ID-module over $S$, $R$ a solution ring for $M$, and $\vect{e}=(e_1,\dots, e_r)$ be an $R$-basis of $R\otimes_S M$ consisting of ID-constant elements. Furthermore,
let $x_1,\dots, x_l\in S$ such that $\gener{x_1,\dots, x_l}_S=S$ and $M[\frac{1}{x_i}]$ is free over $S[\frac{1}{x_i}]$ for all $i=1,\dots, l$.\footnote{The $x_i$ exist, since $M$ is projective by Theorem \ref{thm:M-projective}, hence locally free in the stronger sense.} For all $i$ let $\vect{b_i}$ be a basis of $M[\frac{1}{x_i}]$ over $S[\frac{1}{x_i}]$ consisting of elements in $M$, and $Y_i\in \Mat_{r\times r}(R)$ s.t.~$\vect{b_i}=\vect{e}Y_i$ ($i=1,\dots,l$).
Furthermore, choose $n_i\in\NN$ such that $x_i^{n_i}M\subseteq \gener{\vect{b_i}}_S$.

Then there is a unique Picard-Vessiot ring $\tR$ for $M$ inside $R$, and it is explicitely given by 
$\tR:=S[Y_j, \det(x_j^{n_j}Y_j^{-1})\mid j=1,\dots l]$.
\end{thm}

\begin{proof}
First at all, since $\gener{x_1,\dots, x_l}_S=S$ and therefore $\gener{x_1^{n_1},\dots, x_l^{n_l}}_S=S$, there exist $a_1,\dots, a_l\in S$ s.t.~$\sum_{i=1}^l a_ix_i^{n_i}=1$. This also implies that $\vect{b_1}\cup\dots \cup \vect{b_l}$ is a set of generators for $M$.
Hence, $\vect{b_1}\cup\dots \cup \vect{b_l}$ is a set of generators for $R\otimes_S M$, and there also is a matrix $\tilde{Y}\in \Mat_{rl\times r}(R)$ s.t.~$\vect{e}=(\vect{b_1},\dots, \vect{b_l})\tilde{Y}$.
The proof now procedes in three steps:

\textbf{Step 1:} We show that $\tR=S[Y_j, \det(x_j^{n_j}Y_j^{-1})\mid j=1,\dots l]\subseteq R$:\\
Since $x_j^{n_j}M\subseteq \gener{\vect{b_j}}_S$ and $\vect{b_i}x_j^{n_j}=\vect{b_j}x_j^{n_j}Y_j^{-1}Y_i$, the matrix $(x_j^{n_j}Y_j^{-1}Y_i)$ has coefficients in $S$ for all $i,j$. Then
\begin{eqnarray*}
\vect{e}x_j^{n_j}&=&(\vect{b_1},\dots, \vect{b_l})\tilde{Y}x_j^{n_j} =(\vect{b_1}x_j^{n_j},\dots, \vect{b_l}x_j^{n_j})\tilde{Y}\\
&=& \vect{b_j}(x_j^{n_j}Y_j^{-1}Y_1,\dots, x_j^{n_j}Y_j^{-1}Y_l)\tilde{Y} \in \gener{\vect{b_j}}_R
\end{eqnarray*}
and $\vect{e}x_j^{n_j}=\vect{b_j}(x_j^{n_j}Y_j^{-1})$. Therefore, $x_j^{n_j}Y_j^{-1}\in \Mat_{r\times r}(R)$. Therefore we obtain $\tR=S[Y_j, \det(x_j^{n_j}Y_j^{-1})]\subseteq R$.

\textbf{Step 2:} $\tR$ is a solution ring for $M$:\\
$\tR$ is ID-simple, since all the localisations $\tR[\frac{1}{x_i}]$ are  ID-simple by the consideration of the special case of a free ID-module, because they are just Picard-Vessiot rings for the free $S[\frac{1}{x_i}]$-modules $M[\frac{1}{x_i}]$. ($Y_i^{-1}$ is a fundamental solution matrix for $M[\frac{1}{x_i}]$.)\\
Furthermore, $\tR\otimes_S M$ contains the basis $\vect{e}$, since
$$\vect{e}=\vect{e}\cdot \sum_{j=1}^l a_jx_j^{n_j}=\sum_{j=1}^l \vect{b_j}(x_j^{n_j}Y_j^{-1})a_j\in \gener{\vect{b_1}\cup \dots\cup \vect{b_l}}_{\tR}=\tR\otimes_S M.$$
Hence, $\tR\otimes_S M$ is a trivial ID-module and therefore $\tR$ is a solution ring for~$M$.

\textbf{Step 3:} $\tR$ is a Picard-Vessiot ring for $M$, and the unique one inside $R$:\\
The steps 1 and 2 work for any solution ring $R$, in particular for a Picard-Vessiot ring $R'\subseteq R$. In this case, by minimality of $R'$, and $\tR\subseteq R'$, we obtain that $R'=\tR$. Therefore $\tR$ is a Picard-Vessiot ring, and the unique one inside $R$.
\end{proof}

\begin{cor}\label{cor:faithful-flatness}
Let  $(R,\theta_R)$ be a Picard-Vessiot ring for $M$. Then:
\begin{enumerate}
\item[(a)] All $r\in R$ are ID-finite over $S$.
\item[(b)] $R/S$ is faithfully flat.
\end{enumerate}
\end{cor}

\begin{proof}
\textbf{(a)} By the previous theorem, $R=S[Y_j, \det(x_j^{n_j}Y_j^{-1})\mid j=1,\dots l]$, using the same notation as in the theorem. By Cor.~\ref{cor:ID-finite-subalgebra}, the set of ID-finite elements is a subalgebra of $R$, so we only have to show that the  entries of $Y_j$ and $\det(x_j^{n_j}Y_j^{-1})$ are ID-finite over $S$:\\
Since $\vect{b_1}\cup\dots \cup \vect{b_l}$ is a generating set of $M$, there is a matrix $A\in \Mat_{rl\times rl}(S[[T]])$ s.t.~
$\theta(\vect{b_1},\dots, \vect{b_l})=(\vect{b_1},\dots, \vect{b_l})\cdot A$. Hence,
$$\vect{e}\theta(Y_1,\dots, Y_l)=\theta(\vect{b_1},\dots, \vect{b_l})= \vect{e}(Y_1,\dots, Y_l)\cdot A,$$
which implies $\theta(Y_1,\dots, Y_l)=(Y_1,\dots, Y_l)\cdot A$.
Therefore, all entries of the $Y_j$ are ID-finite over $S$.\\
In the local case over $S[\frac{1}{x_i}]$, we get $\theta(Y_i)=Y_i A_i$ for some $A_i\in  \Mat_{r\times r}((S[\frac{1}{x_i}])[[T]])$, and hence $\theta(\det(Y_i))\in \det(Y_i)\cdot (S[\frac{1}{x_i}])[[T]]$, as well as
$\theta(\det(Y_i^{-1}))\in \det(Y_i^{-1})\cdot (S[\frac{1}{x_i}])[[T]]$.

Since $(x_j^{n_j}Y_j^{-1}Y_i)$ has coefficients in $S$ for all $i,j$ (comp.~proof of Thm.~\ref{thm:explicite-pv-ring}), we obtain that
$\det(x_j^{n_j}Y_j^{-1})=\det(x_j^{n_j}Y_j^{-1}Y_i)\det(Y_i^{-1})$ fulfills 
$\theta(\det(x_j^{n_j}Y_j^{-1}))\in \det(x_j^{n_j}Y_j^{-1})\cdot (S[\frac{1}{x_i}])[[T]]$ for all $i$, and hence
 $\theta(\det(x_j^{n_j}Y_j^{-1}))\in \det(x_j^{n_j}Y_j^{-1})\cdot S[[T]]$. Therefore, $\det(x_j^{n_j}Y_j^{-1})$ is also ID-finite over $S$.

\medskip

\textbf{(b)}
By part (a), all elements in $R$ are ID-finite over $S$ and therefore, $R$ is the union of ID-modules over $S$. In particular, $R$ is the union of projective (hence flat) $S$-modules.
Therefore, $R$ is also a flat $S$-module.
%{\it Explicitly: $J\ideal S$ ideal, and $\sum a_i\otimes r_i\in J\otimes_S R$ s.t. $\sum a_ir_i=0\in S\otimes_S R=R$. Take $N\supseteq \gener{r_i}_S$ flat submodule of $R$, then $J\otimes_S N\to N\to R$ is a composition of inclusions and $\sum a_i\otimes r_i\in J\otimes_S N$ has image $0$. So $\sum a_i\otimes r_i=0\in  J\otimes_S N$ and in $ J\otimes_S R$.}

Assume $R/S$ is not faithfully flat. Then there is a maximal ideal $\m\ideal S$ such that $\m R=R$. Hence, there exist $a_i\in \m$, $r_i\in R$ s.t.~$1=\sum_{i=1}^k a_ir_i$. Since, all elements of $R$ are ID-finite, there is an ID-$S$-module $N\subseteq R$ containing all $r_i$. The $S$-submodule $S\cdot 1\subseteq N$ is an ID-submodule (since $1$ is constant). As $N/(S\cdot 1)$ is an ID-module over $S$ and all ID-modules are projective, $S\cdot 1$ is a direct summand of $N$ as $S$-modules. This however contradicts $1\in \m N$.
\end{proof}

\subsection{Existence and uniqueness of Picard-Vessiot rings} \

In Picard-Vessiot theory of ID-modules over ID-fields, it is well known that a Picard-Vessiot ring exists and is unique up to ID-isomorphism if the field of constants is algebraically closed (cf.~\cite[Lemma 3.4]{bhm-mvdp:ideac} resp.~\cite[Prop.~1.20]{mvdp-mfs:gtlde}). If the field of constants is not algebraically closed a Picard-Vessiot ring may not exist (see \cite{as:cpvthlde}) and if it exists, it may not be unique (cf. \cite{td:tipdgtfrz}). But in \cite{td:tipdgtfrz}, Dyckerhoff also gave a criterion for the existence in characteristic zero. This is about the same criterion as we give in Thm.~\ref{thm:pv-ring-exists}. Our proof however is different, and works in arbitrary characteristic.

The existence and uniqueness result for algebraically closed constants is also present in our situation:

\begin{thm}\label{thm:c-alg-closed}
Let $M$ be an ID-module over $S$. If the constants $C$ of $S$ are algebraically closed, then there exists a Picard-Vessiot ring for $M$ and it is unique up to ID-isomorphism.
\end{thm}

\begin{proof}
By \cite[Lemma 3.4]{bhm-mvdp:ideac} there exists a PV-ring $\tR$ for $\Quot(S)\otimes_S M$ over $\Quot(S)$. In particular, $\tR$ is a solution ring for $M$. Hence by Thm.~\ref{thm:explicite-pv-ring}, there is a PV-ring $R$ for $M$ inside $\tR$. Uniqueness of the PV-ring $R$ can be seen as follows: For a PV-ring $R$, the ID-ring $\tR:=\Quot(S)\otimes_S R$ is a PV-ring for $\Quot(S)\otimes_S M$ over $\Quot(S)$. Since $\tR$ is unique for $\Quot(S)\otimes_S M$ up to ID-isomorphism by  \cite[Lemma 3.4]{bhm-mvdp:ideac}, and $R$ is unique inside $\tR$ by Thm.~\ref{thm:explicite-pv-ring}, $R$ is the unique PV-ring for $M$ up to ID-isomorphism.\\
The uniqueness could also be deduced directly from Theorem \ref{thm:the-scheme-isom}, since for algebraically closed constants $C$, the affine scheme $\uIsom_S^{ID}(R,R')$ always has a $C$-rational point ($R$ and $R'$ being PV-rings for $M$).
\end{proof}

\begin{thm}\label{thm:pv-ring-exists}
Let $S$ have a maximal ideal $\m$ satisfying $S/\m\isom C=C_S$. Then for any ID-module $M$ over $S$ there exists a Picard-Vessiot ring $R$ for $M$.
\end{thm}

\begin{proof}
Since $S$ is ID-simple and has a maximal ideal $\m$ satisfying $S/\m\isom C$, $(S,\theta)$ can be embedded into $(C[[t]],\theta_t)$  by Theorem \ref{thm:embedding}.  Since $C[[t]]$ is ID-simple, has constants $C$ (comp.~Ex.~\ref{ex:ID-rings}\ref{item:der by t}) and $C[[t]]\otimes_S M$ is a trivial ID-module by Theorem \ref{thm:trivial-over-ct},
 $C[[t]]$ therefore is a solution ring for $M$. Hence by Theorem \ref{thm:explicite-pv-ring}, there exists a Picard-Vessiot ring for $M$ (even unique inside $C[[t]]$ with respect to this given embedding of $S$ in $C[[t]]$).
\end{proof}

%----------------------------------------
%----------------------------------------
\section{The differential Galois group scheme}

In this section we establish the Galois correspondence for a Picard-Vessiot extension analogous to the classical ones. The ideas are the same as in \cite[Sect.~2]{td:tipdgtfrz} resp.~in \cite[Sect.~10/11]{am:gticngg}. But we have to do a bit more work, since our modules are not free.

\smallskip

As in the previous section, let $C$ denote an arbitrary field, $(S,\theta)$ an ID-simple ring with $C_S=C$, and let $(M,\theta_M)$ denote an ID-module over $S$.

\begin{thm}\label{thm:torsor-isomorphism}
Let $R'$ be a solution ring for $M$ and $R$ a PV-ring for $M$. Then the map
$$\alpha:R'\otimes_C C_{R'\otimes_S R}\longrightarrow R'\otimes_S R, r\otimes a\mapsto (r\otimes 1)\cdot a$$
is an ID-isomorphism. Furthermore, $ C_{R'\otimes_S R}$ is a finitely generated $C$-algebra.
\end{thm}

\begin{proof}
By definition $\alpha$ is an ID-homomorphism.\\
First we show injectivity: Since $\alpha$ is an ID-homomorphism, $\Ker(\alpha)$ is an ID-ideal of $R'\otimes_C C_{R'\otimes_S R}$. Since $R'$ is ID-simple, $\Ker(\alpha)$ is generated by elements in $C_{R'\otimes_S R}$ by Proposition \ref{prop:constants-of-triv-ext}. But $C_{R'\otimes_S R}$ embeds into $R'\otimes_S R$. Hence, $\Ker(\alpha)=\{0\}$.\\
For showing surjectivity, we use the notation of Theorem \ref{thm:explicite-pv-ring}. So let $x_1,\dots ,x_l\in S$ be such that $\gener{x_1,\dots, x_l}_S=S$ and $M[\frac{1}{x_i}]$ is free as $S[\frac{1}{x_i}]$-module for all $i=1,\dots, l$, and let $\vect{b_i}$ be a basis of $M[\frac{1}{x_i}]$ consisting of elements of $M$, and $n_i\in\NN$ such that $x_i^{n_i}M\subseteq \gener{\vect{b_i}}_S$.
Furthermore, let $\vect{e}$ resp. $\vect{e}'$ be a basis of constant elements in $R\otimes_S M$ resp. $R'\otimes_S M$. Additionally, let $Y_i\in \Mat_r(R)$ and $X_i\in \Mat_r(R')$ such that $\vect{b_i}=\vect{e}Y_i=\vect{e}'X_i$.
Then $R$ is generated over $S$ by the entries of $Y_i$ and $x_i^{n_i}Y_i^{-1}$, and by $R'$-linearity of $\alpha$ it is enough to show that these entries are in $\Ima(\alpha)$.

$\vect{e}$ and $\vect{e}'$ can also be viewed as bases of the free $(R'\otimes_S R)$-module $(R'\otimes_S R)\otimes_S M$.\footnote{$R$ and $R'$ both embed into $R'\otimes_S R$, since they are both ID-simple.}
Hence, there is a matrix $Z\in \GL_r(R'\otimes_S R)$ such that $\vect{e} Z=\vect{e}'$. Since both $\vect{e}$ and $\vect{e}'$ consist of constant vectors the entries of $Z$ are also constant, and the same holds for its inverse $Z^{-1}\in \GL_r(R'\otimes_S R)$. Hence, $Z\in  \GL_r(C_{R'\otimes_S R})$.

For all $i$ we have $\vect{e}Y_i=\vect{b_i}=\vect{e}'X_i=\vect{e}ZX_i$ and hence $Y_i=ZX_i\in \Mat_r(R'\otimes_S R)$, as well as $x_i^{n_i}Y_i^{-1}=(x_i^{n_i}X_i^{-1})Z^{-1}$.
Hence, the entries of all $Y_i$ and of all $x_i^{n_i}Y_i^{-1}$ are in the image of $\alpha$.

\smallskip

Finally, as just seen, the restriction of $\alpha$ to $R'\otimes_C C[Z,Z^{-1}]\subseteq R'\otimes_C C_{R'\otimes_S R}$ is also surjective. But $\alpha$ is an isomorphism, so \mbox{$R'\otimes_C C[Z,Z^{-1}]=R'\otimes_C C_{R'\otimes_S R}$}. Therefore, $C[Z,Z^{-1}]=C_{R'\otimes_S R}$, and $C_{R'\otimes_S R}$ is a finitely generated $C$-algebra.
\end{proof}

\begin{prop}
Let $R$ and $R'$ be PV-rings for $M$, and let $D$ be a $C$-algebra equipped with the trivial iterative derivation.
Then any $(S\otimes_C D)$-linear ID-homomorphism $R\otimes_C D\to R'\otimes_C D$ is an isomorphism.
\end{prop}

\begin{proof}
Let $\beta:R\otimes_C D\to R'\otimes_C D$ be an $(S\otimes_C D)$-linear ID-homomorphism. As in the previous proof, $\Ker(\beta)$ is generated by constants and hence is trivial.
For the surjectivity, we remark that $\beta(R)$ and $R'$ are both PV-rings for $M$.
As in the previous proof, there are bases of constant elements $\vect{e}$ and $\vect{e}'$ in $\beta(R)\otimes_S M$ resp. $R'\otimes_S M$ which can both be viewed as bases of the free $(R'\otimes_C D)$-module  $(R'\otimes_C D)\otimes_S M$.

The same arguments as in the previous proof (with $R$ and $R'$ switched) show that $R'$ is contained in the subring $\beta(R\otimes_C D)=\beta(R)\cdot D$ of $R'\otimes_C D$. Hence by $D$-linearity $\beta$ is surjective.
\end{proof}

\begin{thm}\label{thm:the-scheme-isom}
Let $R$ and $R'$ be PV-rings for $M$. Then the functor
$$\uIsom_S^{ID}(R,R'): (\cat{Algebras}/C) \longrightarrow (\cat{Sets}), D\mapsto \Isom_S^{ID}(R\otimes_C D,R'\otimes_C D)$$
is represented by $\Spec(C_{R'\otimes_S R})$. In particular, it is an affine scheme of finite type over~$C$.
\end{thm}

\begin{proof}
Using the previous proposition and theorem, the proof is exactly the same as in \cite[Cor.~2.11]{td:tipdgtfrz}, or in \cite[Prop.~10.9.]{am:gticngg}.
\end{proof}

As a special case for $R'=R$ we obtain the representability of $\underline{\Aut}^{ID}(R/S)$.

\begin{cor}\label{cor:galois-group-scheme}
For a PV-extension $R/S$, the group functor $\underline{\Aut}^{ID}(R/S)$ is represented by $\G:=\Spec(C_{R\otimes_S R})$, and thus $\G\isom \underline{\Aut}^{ID}(R/S)$ is an affine group scheme of finite type over $C$.
\end{cor}

\begin{defn}
We call $\G=\underline{\Aut}^{ID}(R/S)$ the {\markdef ID-Galois group (scheme)} of $R/S$ and denote it by $\uGal(R/S)$.
\end{defn}

\begin{prop}
Let $R/S$ be a PV-extension and $\G=\uGal(R/S)$ the ID-Galois group. Denote $\G_S:=\G\times_C \Spec(S)$ the extension of $\G$ by scalars.  Then $\Spec(R)$ is a $\G_S$-torsor. 
\end{prop}

\begin{proof}
The inverse of the isomorphism $\alpha$ of Theorem \ref{thm:torsor-isomorphism} for $R'=R$ induces an isomorphism of affine schemes
$$ \G_S\times_{\Spec(S)} \Spec(R)=\G\times_C \Spec(R) \longrightarrow \Spec(R)\times_{\Spec(S)} \Spec(R).$$
By bookkeeping of the identifications one verifies that this map is indeed the isomorphism $(g,x)\mapsto (g(x),x)$ indicating that $\Spec(R)$ is a $\G_S$-torsor.
\end{proof}

For setting up a Galois correspondence, we need a definition of functorial invariants (comp.~\cite[Sect.~11]{am:gticngg}):

Let $\H \leq \G$ be a subgroup functor, \ie{} for every
$C$-algebra $D$, the set $\H(D)$ is a group acting on $R_D:=R\otimes_C
D$ and this action is functorial in $D$. 
An element $r\in R$ is then called {\markdef invariant} under $\H$ if
for all $D$, the element $r\otimes 1\in R_D$ is invariant under
$\H(D)$. The ring of invariants is denoted by $R^{\H}$. (In
\cite[I.2.10]{jcj:rag} the invariant elements are called ``fixed
points''.)

\begin{rem}\label{rem:rho}
Let $\gamma:R\otimes_S R\to R\otimes_C C[\G]$ denote the inverse of the isomorphism $\alpha$.
The action of $\G:=\underline{\Gal}(R/S)$ on $R$ is fully described by the
ID-homomorphism $\rho:=\gamma\restr{1\otimes R}:R\to R\otimes_C C[\G]$.
Namely, for a $C$-algebra $D$ and $g\in \G(D)$ with corresponding
$\tilde{g}\in \Hom(C[\G],D)$, one
has $g(r\otimes 1)=(1\otimes \tilde{g})(\rho(r))\in R\otimes_C D$ for all
$r\in R$.\\
Furthermore, for a closed subgroup scheme $\H\leq \G$, defined by an ideal $I\subseteq C[\G]$,  one has $r\in R^\H$ if and only if, $\rho(r)\equiv r\otimes 1 \mod{R\otimes I}$, resp.~if $\pi^\G_\H( \rho(r))= r\otimes 1\in R\otimes_C C[\H]$ where $\pi^\G_\H:C[\G]\to C[\G]/I=C[\H]$ denotes the canonical projection.
\end{rem}

\begin{prop}\label{prop:R^G=S}
Let $R/S$ be a PV-extension and $\G=\uGal(R/S)$ the ID-Galois group. Then $R^\G=S$.
\end{prop}

\begin{proof}
By the previous remark, $r\in R^\G$ if and only if $\rho(r)=r\otimes 1$. This means that $\gamma(r\otimes 1)=r\otimes 1=\gamma(1\otimes r)$ which is equivalent to $r\in \Quot(S)$. Since, $R/S$ is faithfully flat by Cor.~ \ref{cor:faithful-flatness}, we obtain $r\in S$.
\end{proof}

\begin{rem}\label{rem:converse-to-R^G=S}
The converse conclusion to Proposition \ref{prop:R^G=S}, i.e.~ that $R^\H=S$ for $\H\leq \G$ implies $\H=\G$, is not true. 
 For example, if $\G=\GL_n(C)$ and $\H$ is taken to be a Borel subgroup, then $R\isom S[\GL_n]$ by Hilbert 90, and $R^\H\isom S[\GL_n]^\H=S$. The geometrical reason is that $R^\H$ is the ring of global sections of the scheme $\Spec(R)/\H$. In case of $\H$ being the Borel subgroup this is a projective scheme over $S$.
\end{rem}

Before we come to the Galois correspondence, we need some lemmas and propositions. We start with a condition on an ID-simple ring ensuring that it is a PV-ring.

\begin{lem}[analog of {\cite[Prop.~10.12]{am:gticngg}}]\label{lem:torsor-is-pv-ring}
Let $R/S$ be a faithfully flat extension of ID-simple rings with $C_R=C_S=C$. Assume there exists an affine group scheme $\G$ of finite type over $C$ such that $\Spec(R)$ is a $\G_S$-torsor and the corresponding isomorphism of $S$-algebras $\gamma: R\otimes_S R\to R\otimes_C C[\G]$ is an ID-isomorphism. Here, as usual $C[\G]$ is equipped with the trivial iterative derivation.
Then $R$ is a Picard-Vessiot ring over $S$.
\end{lem}

\begin{proof}
The proof goes similar to \cite{mt:haapvt}, proof of Thm.~3.3(a)$\Rightarrow$(b).

Since, $\Spec(R)$ is a $\G_S$-torsor, $R$ is finitely generated over $S$, and we can choose $C$-linear independent elements $u_1,\dots, u_r\in R$  such that $R$ is generated over $S$ by these elements. By possibly increasing the set of $u$'s we can assume that 
$\rho(\gener{u_1,\dots, u_r}_C)\subseteq \gener{u_1,\dots, u_r}_C\otimes_C C[\G]$, since by general theory on comodules, every element is contained in a finite dimensional comodule (cf.~\cite[Thm.~2.1.3]{mes:ha}).\footnote{As in Remark \ref{rem:rho}, $\rho:=\gamma|_{1\otimes R}$ denotes the coaction of $C[\G]$ on $R$ corresponding to the action of $\G$.}
Then there are $b_{ij}\in C[\G]$ ($i,j=1,\dots, r$) such that $\rho(u_j)=\sum_{i=1}^r u_i\otimes b_{ij}$ for all $j=1,\dots, r$, or written in matrix notation:
$$\rho(u_1,\dots, u_r)=(u_1,\dots, u_r)\otimes B,$$
for $B=(b_{ij})_{1\leq i,j\leq r}$.\\
Since, $\rho$ is a homomorphism of ID-rings, we also obtain for all $n\in\NN$ that
$$\rho(\theta^{(n)}(u_1),\dots, \theta^{(n)}(u_r))=(\theta^{(n)}(u_1),\dots, \theta^{(n)}(u_r))\otimes B.$$

Now, let $M\subseteq R^r$ be the $S$-module generated by all vectors $(\theta^{(n)}(u_1),\dots, \theta^{(n)}(u_r))$ ($n\geq 0$). Then $M$ is an ID-stable subset of $R^r$ by definition and an $S$-module.

We will show that $M$ is finitely generated as $S$-module, that $R\otimes_S M=R\cdot M=R^r$, as well as that for any $\tilde{R}\subsetneqq R$, the standard basis of $R^r$ is not contained in $\tilde{R}\otimes_S M$.\\
The first shows that $M$ is indeed an ID-module over $S$, the second that $R$ is a solution ring for $M$, and the third that $R$ is a minimal solution ring, hence a Picard-Vessiot ring for $M$.

We consider the matrices 
$$W(k_1,\dots, k_r):=\begin{pmatrix} \theta^{(k_1)}(u_1) & \dots & \theta^{(k_1)}(u_r) \\
\vdots & & \vdots \\  \theta^{(k_r)}(u_1) & \dots & \theta^{(k_r)}(u_r) \end{pmatrix}\in \Mat_{r\times r}(R)$$
for $(k_1,\dots, k_r)\in \NN^r$, and the ideal
$$I:=\gener{\det(W(k_1,\dots, k_r))\mid (k_1,\dots, k_r)\in \NN^r}_R\subseteq R$$
generated by the determinants of all these matrices.

Since $\{u_1,\dots, u_r\}$ are $C$-linearly independent, $\{\theta(u_1),\dots, \theta(u_r)\}\subseteq R[[T]]$ are $R$-linearly independent (cf.~\cite[Prop.~1.5]{mt:haapvt}), and therefore, there is a matrix $W(k_1,\dots, k_r)$ having full rank. In particular, $I\ne \{0\}$.
Furthermore, using the Leibniz determinant formula and the product rule for iterative derivations one obtains

$$\theta^{(n)}(\det(W(k_1,\dots, k_r))) =\sum_{n_1+\dots +n_r=n} \left(\begin{smallmatrix}k_1+n_1\\[1mm] k_1\end{smallmatrix}\right)\cdots \left(\begin{smallmatrix} k_r+n_r\\[1mm] k_r\end{smallmatrix}\right) \det(W(k_1+n_1,\dots, k_r+n_r)).$$
Hence, $I$ is an ID-ideal, and since $R$ is ID-simple, we have $I=R$.
Therefore, there exist matrices $W_1,\dots, W_l\in \{ W(k_1,\dots, k_r) \mid (k_1,\dots, k_r)\in \NN^r\}$ and $a_1,\dots, a_l\in R$ such that $1=\sum_{i=1}^l a_i\det(W_i)$.
Using the adjugate matrices $W_i^\#$ of the $W_i$ we get
$$\mathds{1}_r=\sum_{i=1}^l a_i\det(W_i)\mathds{1}_r=\sum_{i=1}^l a_i W_i^\# W_i.$$
This means that the standard basis of $R^r$ is obtained as an $R$-linear combination of the rows of the $W_i$, and hence $R\cdot M=R^r$.

For the finite generation of $M$, we take $W=W(k_1,\dots, k_r)$ of full rank, and observe that $\rho(W)=W\otimes B$, $\rho(\det(W))=\det(W)\otimes \det(B)$ and $\rho(W^\#)=(1\otimes B^\#)\cdot (W^\# \otimes 1)$, by using $WW^\#=\det(W)\mathds{1}_r$.
Write $\theta^{(n)}(\vect{u}):=(\theta^{(n)}(u_1),\dots, \theta^{(n)}(u_r))$, then we get
\begin{eqnarray*}
\gamma(\det(W)\otimes \theta^{(n)}(\vect{u})W^\#)&=&
\det(W)( \theta^{(n)}(\vect{u})\otimes B)(1\otimes B^\#) (W^\# \otimes 1) \\
&=& (\det(W) \theta^{(n)}(\vect{u})\otimes \det(B)\mathds{1}_r) (W^\# \otimes 1) \\
&=& \theta^{(n)}(\vect{u})W^\#\det(W)\otimes \det(B) \\
&=& \gamma(\theta^{(n)}(\vect{u})W^\#\otimes \det(W))
\end{eqnarray*}

Since, $\gamma$ is an isomorphism, we have 
$\theta^{(n)}(\vect{u})W^\#\otimes \det(W)=\det(W)\otimes \theta^{(n)}(\vect{u})W^\#,$
and since the tensor product is taken over $S$, each entry of $\theta^{(n)}(\vect{u})W^\#$ is a $\Quot(S)$-multiple of $\det(W)$. So there is a vector $\vect{s}=(s_1,\dots, s_r)\in S^r$ and $t\in S$ such that
$t\cdot \theta^{(n)}(\vect{u})W^\#=\vect{s}\det(W)$, and hence
$$t\cdot \theta^{(n)}(\vect{u})=\vect{s}W.$$
This shows that all the vectors $\theta^{(n)}(\vect{u})$ are $S$-linearly dependent to the rows of $W$. Recalling that the rows of $W$ were $R$-linearly independent, this show that for any $S$-module $N$ generated by vectors $\theta^{(n)}(\vect{u})$ for several $n\in\NN$ containing the rows of $W$, one has $R\otimes_S N=R\cdot N$.\\
Applying this to the $S$-module $M$ and to the $S$-module $N$ generated by the rows of the $W_i$ ($i=1,\dots, l$) above, we see that $R\otimes_S N=R^r=R\otimes_S M$. Hence by faithful flatness of $R/S$, $M=N$ and $M$ is a finitely generated $S$-module.

Finally, we observe that for any solution ring $\tilde{R}$ inside $R$, the standard basis  of $R^r$ must be contained in $\tilde{R}\otimes_S M\subseteq R^r$, as it is a basis of constant vectors. Since $(u_1,\dots, u_r)=\sum_{i=1}^r u_i e_i \in M$, we get that $u_i\in \tilde{R}$. Hence, $\tilde{R}=R$.
\end{proof}

\pagebreak

\begin{prop}\label{prop:equivalent-conditions}
Let $R/S$ be a PV-extension with Galois group scheme $\G$. For an ID-ring $T$ with $S\subseteq T\subseteq R$ the following are equivalent:
\begin{enumerate}
\item $T$ is a Picard-Vessiot ring over $S$ for some ID-module.
\item $T$ is ID-simple and stable under the action of $\G$, i.e.~$\rho(T)\subseteq T\otimes_C C[\G]$.
\item There is a normal subgroup scheme $\H$ of $\G$ such that $T=R^\H$.
\end{enumerate}
If the equivalent conditions are fulfilled, the normal subgroup scheme $\H$ in iii) can be taken to be $\H=\uGal(R/T)$.
\end{prop}

\begin{proof}
i)$\Rightarrow$ii): (cf. \cite[proof of Prop.~3.4]{am:igsidgg})\\
Since, $T$ is a PV-extension over $S$, we obtain a commutative diagram 
\begin{center}$\xymatrix{
T\otimes_S T \ar[r]^-{\isom} \ar[d] & T\otimes_C C[\uGal(T/S)]=T\otimes_C
C_{T\otimes_S T} \ar[d] \\
R\otimes_S R \ar[r]^-{\isom}& R\otimes_C C[\G]= R\otimes_C C_{R\otimes_S R}
}$\end{center}
where the vertical maps are just the inclusions. But this implies $\rho(T)\subseteq T\otimes_C
C_{T\otimes_S T}\subseteq T\otimes_C C[\G]$, i.\,e. $T$ is
stable under the action of $\G$.

ii)$\Rightarrow$iii): Since $R$ also is a PV-ring over $T$ for $T\otimes_S M$, the Galois group $\H:=\uGal(R/T)$ exists, and by Prop.~\ref{prop:R^G=S}, we have $R^\H=T$. The group scheme $\H$ is indeed a closed subgroup scheme of $\G$:
$R\otimes_T R$ is a factor ring of $R\otimes_S R$ by an ID-ideal $I$. Since $\gamma:R\otimes_S R\to R\otimes_C C[\G]$ is an isomorphism, one has $\gamma(I)=R\otimes_C J$ for an ideal $J\ideal C[\G]$ by Prop.~\ref{prop:ideal-bijection}. Hence, $C[\H]=C_{R\otimes_T R}=C_{(R\otimes_S R)/I}=C[\G]/J$.
Furthermore, since $T$ is stable under the $\G$-action, for all $C$-algebras $D$ and $g\in \G(D)$, $h\in \H(D)\subseteq \G(D)$, also $g^{-1}hg$ fixes the elements of $T\otimes_C D$, i.e.~$g^{-1}hg\in \H(D)$.
Hence, $\H$ is a normal subgroup of $\G$.

iii)$\Rightarrow$i): (comp.~\cite[proof of Thm.~11.5(ii)]{am:gticngg})\\
First at all we show that $T=R^\H$ is ID-simple, that $T$ has constants $C$, and that $T/S$ is faithfully flat. If $I\ideal T$ is an ID-ideal, then $RI\ideal R$ is an ID-ideal, and hence, $RI$ equals $\{0\}$ or $R$. So $I=(RI)^\H$ is $\{0\}$ or $R^\H=T$. Furthermore $C=C_S\subseteq C_T\subseteq C_R=C$, hence $C_T=C$. Faithful flatness is clear by the proof of Cor.~\ref{cor:faithful-flatness}, since $T\subseteq R$ consists of ID-finite elements. 

The isomorphism $\gamma:R\otimes_S R\to R\otimes_C C[\G]$ is $\H$-equivariant, considered by the action of $\H$ on the right tensor factor, and hence we get an ID-isomorphism $R\otimes_S R^\H\isom R\otimes_C C[\G]^\H$. Again by taking invariants (this time $\H$ is only acting on the first tensor factor), this isomorphism restricts to an isomorphism
$$R^\H\otimes_S R^\H\isom R^\H\otimes_C C[\G]^\H.$$
Since $\H$ is normal, $\G/\H$ is an affine group scheme with $C[\G/\H]\isom C[\G]^\H$ (cf.~\cite[III,\S 3, Thm.~5.6 and 5.8]{md-pg:ga}). By construction of the isomorphism, it is an ID-isomorphism, and it is the isomorphism corresponding to the map $(\G/\H)_S\times_S \Spec(R^\H) \to \Spec(R^\H)\times_S\Spec(R^\H)$ indicating that $\Spec(R^\H)$ is a $(\G/\H)_S$-torsor. Therefore, by Lemma \ref{lem:torsor-is-pv-ring}, $T=R^\H$ is a PV-ring over~$S$.

The statement on the choice of $\H$ has already been seen in the proof of the implication ii)$\Rightarrow$iii).
\end{proof}

\begin{thm}\label{thm:galois-correspondence} (Galois correspondence)
Let $R/S$ be a PV-extension for some ID-module and $\G=\uGal(R/S)$. Then there is a bijection between
$$\fT:=\{ T \mid S\subseteq T\subseteq R \text{ intermediate PV-ring} \}$$
and
$$\fH:=\{ \H \mid \H\leq \G \text{ closed normal subgroup scheme of }\G \}$$
given by
$\Psi:\fT\to \fH, T\mapsto \uGal(R/T)$ resp.~$\Phi:\fH\to \fT, \H\mapsto R^{\H}$.
\end{thm}

\begin{rem}
The maps $\Psi$ and $\Phi$ can be defined between all intermediate ID-rings and all closed subgroups of $\G$.
But contrary to the Galois correspondences in \cite{am:gticngg} and others, one does not get a full bijection, 
as we only consider the rings and not the fields of fractions. Remark \ref{rem:converse-to-R^G=S} provides an example that the extension $\Phi$ would not be injective in general.\\
Maybe, one would get a full bijection when considering schemes with ID-simple structure sheaves, because $\G/\H$, and therefore $\Spec(R)/\H$ is a non-affine scheme in general.
\end{rem}

\begin{proof}[Proof of Thm.~\ref{thm:galois-correspondence}]
Prop.~\ref{prop:equivalent-conditions} already shows most things: If $\H$ is a normal subgroup scheme of $\G$, then $R^\H$ is a PV-ring. Hence $\Phi$ is welldefined. If $T$ is an intermediate PV-ring, the group scheme $\H:=\uGal(R/T)$ is a closed normal subgroup scheme of $\G$. Hence, $\Psi$ is welldefined. Furthermore, $R^{\uGal(R/T)}=T$ showing $\Phi\circ \Psi=\id$.

It remains to show that $\uGal(R/R^\H)=\H$ for all closed normal subgroup schemes $\H$ of $\G$.
In the proof of Prop.~\ref{prop:equivalent-conditions}, it is shown that $\uGal(R^\H/S)\isom \G/\H$. Furthermore, the projection map $\G\to \G/\H$ corresponds to the map $\uGal(R/S)\to \uGal(R^\H/S)$ given by restricting the automorphisms in $\uGal(R/S)$ to $R^\H$. Hence, $\uGal(R/R^\H)$ is the kernel of this map, i.e.~equals $\H$.
\end{proof}

%---------------------------------------------------------------------
% Ende des eigentlichen Artikels
%---------------------------------------------------------------------
\bibliographystyle{plain}
\def\cprime{$'$}

%\bibliography{../reference}

\vspace*{.5cm}

\end{document}